\documentclass{amsart}
\usepackage{amssymb}
\usepackage{amscd}
\usepackage{amsmath}
\usepackage{verbatim}
\copyrightinfo{2009}{American Mathematical Society}

\newtheorem{theorem}{Theorem}[section]
\newtheorem{lemma}[theorem]{Lemma}
\newtheorem{corollary}[theorem]{Corollary}

\theoremstyle{definition}
\newtheorem{definition}[theorem]{Definition}

\theoremstyle{remark}
\newtheorem{remark}[theorem]{Remark}

\theoremstyle{definition}
\newtheorem*{ack}{Acknowledgments}

\newcommand{\dbar}{\bar\partial}
\newcommand{\im}{{\rm im\,}}

\newcommand{\dom}{{\rm dom\,}}
 
\newcommand{\hooklongrightarrow}{\lhook\joinrel\relbar\joinrel\relbar\joinrel\longrightarrow}

\numberwithin{equation}{section}

\begin{document}

\title[Transversal Fredholm Property]{A transversal Fredholm property for the $\bar\partial$-Neumann problem on $G$-bundles\\
{\tiny\sc Dedicated to M.A.\ Shubin on his $65^{\rm th}$ birthday}}

\author{Joe J Perez}
\address{Universit\"at Wien}
\email{joe\_j\_perez@yahoo.com}
\thanks{Supported by FWF grant P19667, {\it Mapping Problems in Several 
Complex Variables}}

\subjclass[2010]{Primary 32W05; 35H20}

\date{}

\begin{abstract} Let $M$ be a strongly pseudoconvex complex $G$-manifold with compact quotient $M/G$. We provide a simple condition on forms $\alpha$ sufficient for the regular solvability of the equation $\square u=\alpha$ and other problems related to the $\bar\partial$-Neumann problem on $M$.\end{abstract}

\maketitle

\section{Introduction}

Let $M$ be a manifold which is the total space of a $G$-bundle
\[G\longrightarrow M\longrightarrow X\]
\noindent
with $X$ compact. With respect to a $G$-invariant measure on $M$, define the Hilbert space $L^2(M)$. This decomposes as 
\begin{equation}\label{decomp} L^2(M)\cong L^2(G)\otimes L^2(X),\end{equation}
and if we assume that the action of $G$ is from the right, then $t\in G$ acts in $L^2(M)$ by $t \to R_t \otimes {\bf 1}_{L^2(X)}$. The von Neumann algebra of operators on $L^2(G)$ commuting with right translations is denoted by $\mathcal L_G$ and the corresponding algebra of bounded linear operators on $L^2(M)$ that commute with the action of $G$ is denoted by $\mathcal B(L^2(M))^G$. This has a decomposition itself as follows,
\[\mathcal B(L^2(M))^G \cong \mathcal B(L^2(G)\otimes L^2(X))^G \cong \mathcal L_G\otimes \mathcal B(L^2(X))).\]
\begin{definition}Let $M$ be a $G$-manifold with quotient $X=M/G$ and let $\mathcal H_1, \mathcal H_2$ be Hilbert spaces of sections of bundles over $M$. A closed, densely defined, linear operator $A:\mathcal H_1 \to \mathcal H_2$ which commutes with the action of $G$ is called \emph{transversally Fredholm} if the following conditions are satisfied:

\begin{enumerate}

\item there exists a finite-rank projection $P_{L^2(X)}\in\mathcal B(L^2(X))$ such that $\ker A\subset \im({\bf 1}_{L^2(G)}\otimes P_{L^2(X)})$ 

\item there exists a finite-rank projection $P'_{L^2(X)}\in\mathcal B(L^2(X))$ such that $\im A\supset\im({\bf 1}_{L^2(G)}\otimes P'_{L^2(X)})^\perp$.
\end{enumerate}\end{definition}

This note will provide a simple example of this idea. Let $M$ be a strongly pseudoconvex complex manifold which is also the total space of a $G$-bundle $G\longrightarrow M\longrightarrow X$ with $X$ compact. Furthermore, assume that $G$ acts on $M$ by holomorphic transformations. With respect to a $G$-invariant measure and Riemannian structure, define the Hilbert spaces of $(p,q)$-forms $L^2(M,\Lambda^{p,q})$.

On $M$, consider Kohn's Laplacian, $\square$ and its spectral decomposition, $\square = \int_0^\infty \lambda dE_\lambda$ in $L^2(M,\Lambda^{p,q})$. If $q>0$, it was shown in \cite{P1} that if $\delta\ge 0$, then the Schwartz kernel of the spectral projection $P_\delta = \int_0^\delta dE_\lambda$ belongs to $C^\infty(\bar M\times\bar M)$. Choosing a piecewise smooth section $X\hookrightarrow M$, we may write points in $M$ as pairs $(t,x)\in G\times X$. The Schwartz kernel $K$ of $P_\delta$ then, almost everywhere, takes the form
\[K(t,x;s,y) = K(ts^{-1},x;e,y)=:\kappa(ts^{-1};x,y),\]
\noindent
where we have used the $G$-invariance of $P_\delta$. It is also true that $\kappa$ has an expansion
\begin{equation}\label{kappa}\kappa(t;x,y) = \sum_{kl} \psi_k(x) h_{kl}(t)\bar\psi_l(y)\end{equation}
\noindent
where $(\psi_k)_k$ is an orthonormal basis of $L^2(X)$. The functions $h_{kl}$ are smooth in $G$ with $\sum_{kl}\|h_{kl}\|_{L_R^2(G)}^2<\infty$, where $L_R^2(G)$ consists of the functions on $G$ that are square-integrable with respect to right-Haar measure ({\it cf.}\ proof of Lemma 6.2 in \cite{P1}).

The main result of the present paper is the fact that when $\kappa$ corresponds to $P_\delta$, the sum in equation \eqref{kappa} can be taken to be finite. This means that the spectral projections of $\square$ are subordinate to simple projections of the form $P={\bf 1}_{L^2(G)}\otimes P_{L^2(X)}$ with $P_{L^2(X)}$ the projection onto the space spanned by the $\psi_k$ that appear in the sum. Since there are finitely many, we have that ${\rm rank}\, P_{L^2(X)}<\infty$. Thus our main result in this note is
\begin{theorem} Let $M$ be a strongly pseudoconvex complex manifold which is also the total space of a $G$-bundle $G\longrightarrow M\longrightarrow X$ with $X$ compact. Furthermore, assume that $G$ acts on $M$ by holomorphic transformations. It follows that for $q>0$, the Laplacian $\square$ in $L^2(M,\Lambda^{p,q})$ is transversally Fredholm. \end{theorem}
\noindent
We will also show that the $\dbar$-Neumann problem has regular solutions for $g\in \im P^\perp$. 

As well as sharpening the results in \cite{P1}, the results of this note will be useful in studying the $\dbar$-Neumann problem and its consequences for $G$-manifolds with nonunimodular structure group; in \cite{P1}, $G$ was always assumed unimodular. These $G$-manifolds, among others, occur naturally as complexifications of group actions, as shown in \cite{HHK}.

The present results, in addition to the amenability property introduced in \cite{P2}, will lead to a better understanding of two important exemplary nonunimodular $G$-manifolds discussed in \cite{GHS}. One of these has a large space of $L^2$-holomorphic functions while the other has $L^2\mathcal O=\{0\}$. 

\begin{remark}All the results in this note remain valid for weakly pseudoconvex $M$ satisfying a subelliptic estimate, and for the boundary Laplacian, $\square_b$, \cite{P3}. \end{remark}

\section{Invariant operators in $L^2(M)$}\label{decomp}

Here we briefly sketch the construction of the Schwartz kernel \eqref{kappa} of $P_\delta$. We will continue to simplify notation by suppressing the operators' acting in bundles; some additional details are in \cite{P1}. 

On the group alone, the projection $P_L$  onto a translation-invariant subspace $L\subset L^2(G)$ is a left-convolution operator with distributional kernel $\kappa$,
\[(P_L u)(t) = (\lambda_\kappa u)(t) = \int_G ds\  \kappa(ts^{-1})u(s),\qquad (u\in L^2(G)),  \]
\noindent
where $ds$ is the right-invariant Haar measure.

Let us lift this definition to $L^2(M)$ by taking the decomposition \eqref{decomp} a step further. Letting $(\psi_k)_k$ be an orthonormal basis for $L^2(X)$, we may write
\[L^2(M)\cong L^2(G)\otimes L^2(X) \cong \bigoplus_k L^2(G)\otimes\psi_k,\]
\noindent
and with respect to this decomposition write matrix representations for operators in $L^2(M)$ as 
\[\mathcal B(L^2(M)) \ni P \longleftrightarrow [P_{kl}]_{kl}, \qquad P_{kl}\in\mathcal B(L^2(G)).\]
\noindent
When $P\in\mathcal B(L^2(M))^G$ each of the $P_{kl}$ is an operator commuting with the right action and thus is a left convolution operator. Thus $P_{kl}=\lambda_{h_{kl}}$ for distributions $h_{kl}$ on $G$, as in the expansion \eqref{kappa}. When $P$ is a self-adjoint projection, we find that the matrix of convolutions $H=[\lambda_{h_{kl}}]_{kl}$ is an idempotent in that $\sum_k H_{jk} H_{kl} = H_{jk}$ and the matrix corresponding to $P^*$, has matrix representation $[\lambda_{h_{lk}}^*]_{kl}$. 

\section{Regularity of the $\bar\partial$-Neumann problem on $G$-manifolds}\label{regres}

We provide a brief list of the properties of the $\dbar$-Neumann problem relevant to our work here and refer the reader to \cite{FK, GHS, P1} for more detail. With the invariant measure and Riemannian structure on $M$ define the Sobolev spaces $H^s(M,\Lambda^{p,q})$ of $(p,q)$-forms on $M$. Note that the $G$-invariance of the structures and the compactness of $X$ imply that any two such Sobolev spaces are equivalent. A word on notation: we will write $A\lesssim B$ to mean that there exists a $C>0$ such that $|A(u)|\le C |B(u)|$ uniformly for $u$ in a set that will be made clear in the context.
\begin{lemma}\label{reg1} Suppose that $M$ is strongly pseudoconvex and $U$ is an open subset of $\bar M$ with compact closure. Assume also that $\zeta, \zeta_{1}\in C^{\infty}_{c}(U)$ for which $\zeta_{1}|_{{\rm supp}(\zeta)}=1$.  If $q>0$ and $\alpha|_{U}\in H^{s}(U,\Lambda^{p,q})$, then $\zeta(\square +1)^{-1}\alpha\in H^{s+1}(\bar M,\Lambda^{p,q})$ and 
 \begin{equation}\label{est1}\|\zeta (\square +1)^{-1}\alpha\|_{s+1}^2\lesssim \|\zeta_{1}\alpha\|_s^2+\|\alpha\|_0^2.\end{equation}
 \end{lemma}
 \begin{proof}This is Prop. 3.1.1 from \cite{FK} extended to the noncompact case in \cite{E}. \end{proof}

It follows easily (Corollary 4.3, \cite{P1}) that the image of the Laplacian's spectral projection $P_\delta$ is contained in $C^\infty(\bar M, \Lambda^{p,q})$. 

In order to derive properties of the Schwartz kernel of $P_\delta$, we will need global Sobolev estimates strengthening the previous result. The following assertion (Theorem 4.5 of \cite{P1}) provides global {\it a priori} Sobolev estimates on $M$ and is a generalization of Prop.\ 3.1.11, \cite{FK} to the noncompact case. Note that this crucially uses the uniformity on $M$ guaranteed by the $G$-action and the compactness of $X$.

\begin{lemma}\label{reg2} Let $q>0$. For every integer $s\ge 0$, the following estimate holds uniformly,
\[\| u\|_{s+1}^2\lesssim\|\square u\|_s^2+ \| u\|_0^2, \qquad (u\in \dom(\square)\cap C^\infty(\bar M,\Lambda^{p,q})).\]
\end{lemma}
\noindent
The previous two lemmata give

\begin{corollary}\label{bigsmooth} For $q>0$, let $\square = \int_0^\infty \lambda dE_\lambda$ be the spectral decomposition of the Laplacian $\square$ and for $\delta\ge0$, define $P_\delta = \int_0^\delta dE_\lambda$. Then $\im P_\delta\subset H^\infty(M)$.\end{corollary}
\begin{proof}The assertion follows from lemmata \ref{reg1}, \ref{reg2} and the fact that $\im P_\delta\subset\dom\square^k$ for all $k=1,2,\dots$. Thus the estimates
\[\|\square^{k-s}u\|_{s+1}\lesssim \|\square^{k-s+1} u\|_s + \|\square^{k-s}u\|_0, \quad (s=1,2,\dots,k)\]
\noindent
hold for $u\in\im P_\delta$. These can be reduced to the result.\end{proof}
\begin{remark} By results in \cite{E, P3}, these regularity properties essentially hold true for $G$-manifolds $M$ that are weakly pseudoconvex but satisfy a subelliptic estimate. Similar results hold for the boundary Laplacian $\square_b$ as indicated in \cite{P1}.\end{remark}

\section{The finiteness result}

In this section, we modify an ingenious lemma from \cite{GHS}. In the original setting, this lemma asserts that on a regular covering space $\Gamma\to M\to X$, it is true that any closed, invariant subspace $L\subset L^2(M)$ that belongs to some $H^s(M)$ ($s>0$) has the following property. There exists an $N<\infty$ and a $\Gamma$-equivariant injection $P_N$ such that
\[L \stackrel {P_N}{\hooklongrightarrow} L^2(\Gamma)\otimes\mathbb C^N.\] 
\noindent
This result has analogues in \cite{A} and Theorem 8.10, \cite{LL}, gotten by different methods.

Here, we will use essentially the same proof as in \cite{GHS} to obtain a similar result for $G$-bundles. We will need the following 
\begin{definition} For any positive integer $s$, let $H^{0,s}(G\times X) = L^2(G)\otimes H^s(X)$ be the completion of $C_c^\infty (G\times X)$ in the norm defined by
\[\|u\|_{H^{0,s}(G\times X)}^2 = \int_G dt\ \|u(t,\cdot)\|^2_{H^s(X)}.\]
\noindent 
Clearly $\|\cdot\|_{H^{0,s}(G\times X)}\le\|\cdot\|_{H^s(M)}$ and so $H^s(M)\subset H^{0,s}(G\times X)$.\end{definition}

The next two statements in this section follow \cite{GHS} closely. Lemma \ref{rellichesque} is taken verbatim and Theorem \ref{gam} is a small variation on Prop.\ 1.5 of that article.

\begin{lemma}\label{rellichesque} Let $X$ be a compact Riemannian manifold, possibly with boundary and let $(\psi_k)_k$ be any complete orthonormal basis of $L^2(X)$. Then, for all $s>0$ and $\delta>0$ there exists an integer $N>0$ such that for all $u\in H^s(X)$ in the $L^2$-orthogonal complement of $(\psi_k)_1^N$ we have the uniform estimate
\[\|u\|_{L^2(X)}\le\delta\|u\|_{H^s(X)}, \qquad (u\in H^s(X),\ u\perp\psi_k,\ k=1,2,\dots,N).\]
\end{lemma}
\begin{proof} Assuming the contrary, there exist $s>0$ and $\delta>0$ so that for each $N>0$ there is an $u_N\in H^s(X)$ with $\langle u_N,\psi_k\rangle = 0$ for $k = 1,2,\dots,N$ and $\|u_N\|_s< 1/\delta\  \|u_N\|_0$. Without loss of generality we may rescale the $u_N$ to unit length. By Sobolev's compactness theorem, the sequence $(u_N)_N$ is a compact subset of $L^2(X)$. By the requirement that each $u_N$ be orthogonal to $\psi_k$ for $k=1,2,\dots,N$, the sequence converges weakly to zero. This contradicts the choice of normalization.\end{proof}
\begin{theorem}\label{gam} Assume that $G$ is a Lie group and $G\to M \to X$ is a $G$-bundle with compact quotient, $X$. Let $L$ be an $L^2$-closed, $G$-invariant subspace in $H^\infty(M)$, such that for $s\in\mathbb N$ sufficiently large, $L\subset H^s(M)$ and
\begin{equation}\label{soboest} \|u\|_{H^s(M)}\lesssim \|u\|_{L^2(M)}\end{equation}
\noindent
holds uniformly for $u \in L$. Then $L\subset\im ({\bf 1}_{L^2(G)}\otimes P_{L^2(X)})$ where $P_{L^2(X)}$ is a finite-rank projection in $L^2(X)$.\end{theorem}
\begin{proof} First, assume that $M\cong G\times X$ is a trivial bundle. For each fixed $t\in G$, define the {\it slice at} $t$, $S_t = \{(t,x)\in M\mid x\in X\}$, and note that by the trace theorem, the restrictions of functions in $L$ to these slices are in $H^\infty(S_t)$. Note also that the invariance of $L$ implies that all the restrictions $L|_{S_t}$ are identical. At the identity $e\in G$, choose an orthonormal basis $(\psi_j)_j$ for $L^2(S_e)\cong L^2(X)$. Let $L$ satisfy the assumptions of the theorem and define a map $P_N:L\to L^2(G)\otimes\mathbb C^{N}$ by
\[(P_N u)(t) = (u_1(t), u_2(t),\dots,u_N(t)),\]
\noindent
where
\[u_j(t) = \langle u|_{S_t},\psi_j\rangle_{L^2(X)}, \qquad j=1,2,\dots,N.\]
\noindent
We will show that $P_N$ is injective for large $N$. Assume that $u\in L$ and $P_N u = 0$. The smoothness of all the structures implies that $(P_N u)(t)=0$ identically. Lemma \eqref{rellichesque} and invariance imply that there is a $\delta_N>0$ such that
\begin{equation}\label{sliceest}\|u|_{S_t}\|_{L^2(S_t)}^2\le \delta_N^2\|u|_{S_t}\|^2_{H^s(S_t)}, \quad (t\in G).\end{equation}
\noindent
Integrating over $t\in G$ we obtain
\begin{equation}\label{delta}\|u\|_{L^2(M)}^2\le\delta_N^2\|u\|^2_{H^{0,s}(G\times X)} \le\delta_N^2\|u\|^2_{H^{s}(M)}.\end{equation}
\noindent
If this were possible for any $N$, this would contradict the estimate \eqref{soboest} unless $u=0$, since $\delta_N\to 0$ as $N\to\infty$. To obtain the result for a trivial bundle, let $N$ be the least integer for which $P_N$ is injective and choose $N$ elements $v_1, v_2,\dots, v_N \in L$ whose restrictions to $S_e$ are linearly independent. The result for a general bundle follows by a trivialization argument. \end{proof}

\begin{remark} We should note here that the assumptions are redundant. For $L$ to be $L^2$-closed and in $H^\infty(M)$ implies the validity of an estimate \eqref{soboest} for any $s$. 
\end{remark}

%
\begin{corollary}\label{finsum} Let $\square = \int_0^\infty \lambda dE_\lambda$ be the spectral resolution of the Laplacian and for $\delta > 0$ let $P_\delta = \int_0^\delta dE_\lambda$ be a spectral projection. Also choose a piecewise smooth section $x:X\hookrightarrow M$. It follows that $P_\delta$ has a representation
\begin{equation}\label{summa}(P_\delta u)(t,x) = \sum_{kl=1}^N \int_{G\times X} dsdy\ \psi_k(x) h_{kl}(st^{-1})\bar\psi_l(y) u(s,y),\end{equation}
\noindent
where $(\psi_k)_k$ are an orthonormal basis of $L^2(X)$ and $H=[h_{kl}]_{kl}$ is a self-adjoint, idempotent convolution operator in $\bigoplus_{1}^N L^2(G)$ with $h_{kl}\in C^\infty(G)$. Also, 
\[\sum_{kl=1}^N \|h_{kl}\|_{L_R^2(G)}^2 = \sum_{k=1}^N h_{kk}(e)<\infty.\] \end{corollary}
\begin{proof} By Corollary \ref{bigsmooth}, the theorem applies. Apply the Gram-Schmidt procedure to the $(v_k)_1^N$ above, obtaining the $(\psi_k)_1^N$. The decomposition is described in \S \ref{decomp}.\end{proof}
\begin{remark} In the case that $G$ is unimodular, $\sum_{kl} \|h_{kl}\|_{L_R^2(G)}^2 <\infty$ is the same as saying that $P_\delta$ is in the $G$-trace class, which we established in \cite{P1} in the setting in which $M$ is strongly pseudoconvex and in \cite{P3} where $M$ satisfies a subelliptic estimate. The new content of Corollary \ref{finsum} is the finiteness of the sum \eqref{summa}, {\it etc.} This transverse dimension gives a meaningful (though much rougher) measure of the spectral subspaces of $\square$ (and $\square_b$) than the $G$-dimension when $G$ is unimodular, but is also defined when the group is not assumed unimodular as, for example, in \cite{HHK} and in important examples in \cite{GHS}. We should note that \cite{HHK} also deals with the situation in which the $G$-action is only proper, rather than free as we assume here. \end{remark}
%

\section{Applications} 

We will give a version of the solution of the $\dbar$-Neumann problem, for our noncompact $M$. The version valid for $M$ compact, {\it e.g.}\ Prop.\ 3.1.15 of \cite{FK}, is unlikely to remain valid in our setting because the Neumann operator on a noncompact space is usually unbounded.

Let $\square = \int_0^\infty \lambda dE_\lambda$ be the spectral decomposition of the Laplacian on $M$ and for $\delta>0$ put 
\begin{equation}\label{goodspace}L_\delta = \im\int_\delta^\infty dE_\lambda\quad {\rm and}\quad P_\delta = \int_0^\delta dE_\lambda.\end{equation} 

In this section we will show that $\square u = g$, and the $\dbar$-Neumann problem have regular solutions for $g\in L_\delta$. 

\begin{lemma} If $g\in L_\delta\cap C^\infty(\bar M)$, then the solution $u$ of $\square u=g$ is smooth. \end{lemma}
\begin{proof}Let $g\in L_\delta\cap  C^\infty(\bar M)$ and solve $\square u = g$ in $L^2(M)$. Note that $\|u\|_{L^2(M)}\le (1/\delta) \|g\|_{L^2(M)}$. Adding $u$ to both sides of the equation, $(\square + 1) u = g + u$, we obtain that $(\square+1)u = \square u + u = g + u$. Applying $(\square+1)^{-1}$, the real estimate, Lemma \ref{reg1} provides that
\[\|\zeta u\|_{s+1}\lesssim \|\zeta_1(g+u)\|_s + \|g+u\|_0\le \|\zeta_1 g\|_s + \|\zeta_1 u\|_s + \|g+u\|_0.\]
\noindent
Nesting the supports of cutoff functions, concatenating and reducing these estimates for $s=0, 1, \dots$, we obtain that for each positive integer $s$ we have 
\[\|\zeta u\|_{s+1} \lesssim \|\zeta_1 g\|_s + \|g+u\|_0 \le  \|\zeta_1 g\|_s + \left(1+ 1/\delta\right)\|g\|_0 .\] 
\noindent
Thus $u\in C^\infty(\bar M)$ by the Sobolev embedding theorem.\end{proof}

\begin{corollary}\label{delreg} In $L_\delta$, the Laplacian satisfies the genuine estimate 
\[\|u\|_{s+1} \lesssim \|\square u\|_s + \|u\|_0, \qquad (u\in L_\delta).\]
\end{corollary}
\begin{proof} Let $(g_k)_k \subset L_\delta\cap H^\infty$ and $g_k\to g \in H^s(M)$. The previous lemma implies that there exists a sequence $(u_k)_k\subset C^\infty$ solving $\square u_k = g_k$. Lemma \ref{reg2} implies that $\|u_k\|_{s+1}\lesssim \|\square u_k\|_s + \|u_k\|_0$ uniformly in $k$, so $(u_k)_k$ is Cauchy in the $H^{s+1}$ norm. \end{proof}
\begin{lemma}\label{sol} Suppose that $q>0$, $\alpha\in L^2(M,\Lambda^{p,q}),\dbar\alpha=0$, and $\alpha\in L_\delta$. Then there is a unique solution $\phi$ of $\dbar\phi=\alpha$ with $\phi\perp\ker(\dbar)$. If $\alpha\in H^s(\bar M,\Lambda^{p,q})$, then $\phi\in H^s(\bar M,\Lambda^{p,q-1})$ and $\|\phi\|_s\lesssim\|\alpha\|_s$ for each $s$.\end{lemma} 
\begin{proof} Taking $\alpha\in L_\delta$, there is a unique solution to $\square u=\alpha$ orthogonal to the kernel of $\square$; in fact $u\in L_\delta\subset(\ker\square)^\perp$. Since $\dbar\alpha=0$, applying $\dbar$ to
\[\square u = \dbar^*\dbar u+\dbar\dbar^* u  = \alpha\]
gives that $\dbar \dbar^*\dbar u = 0$. This implies that $\langle\dbar \dbar^*\dbar u,\dbar u\rangle = 0$ which is equivalent to $ \| \dbar^*\dbar u\|^2 = 0$. Thus $\dbar\dbar^* u  = \alpha$ and we may take $\phi = \dbar^* u \in\im \dbar^*$. But  $\im \dbar^* \subset (\ker\dbar)^\perp$. The regularity claim follows immediately from Corollary \ref{delreg} and the order of $\dbar^*$. \end{proof} 

\noindent
Putting all these results together, we obtain

\begin{corollary} Let $M$ be a complex manifold on which a subelliptic estimate holds. Assume also that $M$ is the total space of a bundle $G\to M\to X$ with $G$ a Lie group acting by holomorphic transformations with compact quotient $X=M/G$. With respect to a piecewise smooth section $X\hookrightarrow M$, define the slices $S_t$. Then there exists a finite-dimensional subspace $L|_{S_e}\subset L^2(X)$, such that the equation $\square u = \alpha$ has solutions $u\in L^2(M)$ with uniform estimates on the space of $\alpha$ satisfying $\alpha|_{S_t}\perp L|_{S_e}$ for all $t\in G$.\end{corollary}
\begin{proof}Choose $\delta>0$. Corollary \ref{bigsmooth} and Theorem \ref{gam} imply that there exists a finite rank projection $P_{L^2(X)}\in\mathcal B(L^2(X))$ such that $P_\delta < {\bf 1}_{L^2(G)}\otimes P_{L^2(X)}$. The orthogonal complement of the latter projection is ${\bf 1}_{L^2(G)}\otimes P_{L^2(X)}^\perp$, which contains $L_\delta$, on which the $\dbar$-Neumann problem is regular by the results of this section. Putting $L|_{S_e} = \im P_{L^2(X)}$, we have the result.\end{proof}

\begin{remark} A similar result holds for the $\dbar$-equation by Lemma \ref{sol}.\end{remark}

\begin{ack} The author wishes to thank Indira Chatterji and Bernhard Lamel for helpful conversations and the Erwin Schr\"odinger Institute for its generous hospitality. \end{ack}

\bibliographystyle{amsplain}

\begin{thebibliography}{XXX}
%
\bibitem[A]{A} Atiyah, M.F.: Elliptic operators, discrete groups, and von Neumann algebras, {\it Soc. Math. de France, Ast\'erisque} {\bf 32-3}, (1976) 43--72
%
\bibitem[E]{E} Engli\v s, M.: Pseudolocal estimates for $\dbar$ on general pseudoconvex domains, {\it Indiana Univ.\ Math.\ J.}\ {\bf 50}, (2001) 1593--1607
%
\bibitem[FK]{FK} Folland, G.B.\ \& Kohn J.J.:  The Neumann Problem for the Cauchy-Riemann Complex,  {\it Ann.\ Math.\ Studies}, No.\ 75 Princeton University Press, Princeton, N.J.\ 1972
%
\bibitem[GHS]{GHS} Gromov, M., Henkin, G.\ \& Shubin, M.: Holomorphic $L^2$ functions on coverings of pseudoconvex manifolds, {\it Geom.\ Funct.\ Anal.}\ {\bf 8}, (1998) 552--585
%
\bibitem[HHK]{HHK} Heinzner, P., Huckleberry, A.\ T., Kutzschebauch, F.: Abels' theorem in the real analytic case and applications to complexifications. In: {\it Complex Analysis and Geometry}, Lecture Notes in Pure and Applied Mathematics, Marcel Dekker 1995, 229--273
%
\bibitem[LL]{LL} Lieb, E.H., Loss, M.: {\it Analysis (Graduate Studies in Mathematics)} American Mathematical Society; 2 ed (2001)
%
\bibitem[P1]{P1} Perez, J.J.: The $G$-Fredholm property for the $\dbar$-Neumann problem, {\it J. Geom.\ Anal.}\ (2009) {\bf 19}: 87--106 
%
\bibitem[P2]{P2} Perez, J.J.: The Levi problem on strongly pseudoconvex $G$-bundles, {\it Ann.\ Glob.\ Anal.\ Geom.}\ (2010) {\bf 37} 1--20
%
\bibitem[P3]{P3} Perez, J.J.: Subelliptic boundary value problems and the $G$-Fredholm property, \verb#http://arxiv.org/abs/0909.1476#
%
\end{thebibliography}

\end{document}